\documentclass[12pt]{article}
\usepackage{amssymb}
\usepackage[mathscr]{eucal}
\usepackage{amsthm,amsmath,anysize}
\def\a{\alpha}

\def\l{\lambda}
\def\g{\gamma}
\def\0{\bar{0}}
\def\1{\bar{1}}

\def\g{\mathfrak{g}}

\def\g{{\mathfrak g}}

\newtheorem{lemma}{Lemma}[section]
\newtheorem{theorem}[lemma]{Theorem}

\newtheorem{corollary}[lemma]{Corollary}

\hoffset \voffset \oddsidemargin=60pt \evensidemargin=45pt
\title{ On the simplicity of  induced modules for  reductive Lie algebras. \rm{II}}
\author{Chaowen Zhang
\\ Department of
Mathematics,\\ China University
 of Mining and Technology,\\ Xuzhou, 221116, Jiang Su, P. R. China}

\date{ }

\begin{document}
\maketitle

\par

\section{Introduction}
  Let $G$ be a reductive algebraic group defined over an algebraically closed field $\mathbf F$ of positive characteristic $p$, and let $\g$ be the Lie algebra of $G$.  In \cite[5.1]{fp2}, Friedlander and Parshall  asked to find necessary and sufficient conditions for the simplicity of a $\g$-module with $p$-character $\chi\in \g^*$ that is induced from a simple module for a parabolic subalgebra of $\g$.  This question has been answered (by V. Kac) when $\g$ is of type $A_2$ (see \cite[Example 3.6]{fp2} and \cite{kac}), and also when $\g$ is of type $A_3$ (see \cite{lsy}).  When $\g$ is of type $A_n$, $B_n$, $C_n$ or $D_n$ and when $\chi$ is of standard Levi form, a sufficient condition is given in \cite{lsy2} for the simplicity of above-mentioned $\g$-modules.
  Under certain assumption of $\chi$, a necessary and sufficient condition for the simplicity of these $\g$-modules
  is given in \cite{zh}. In this paper we give a sufficient condition for the simplicity of these $\g$-modules
  with a  general assumption of $\chi$.    \par
   Following \cite[6.3]{j} we make the following hypotheses:\par
   (H1) The derived group $DG$ of $G$ is simply connected;\par
   (H2) The prime $p$ is good for $\g$;\par
   (H3) There exists a $G$-invariant non-degenerate bilinear form on $\g$.\par
   Let $T$ be a maximal torus of $G$, let $\mathfrak h=\mbox{Lie}(T)$, and let $\Phi$ be the root system of $G$.  Let $\Pi=\{\a_1, \dots,\a_l\}$ be a base of $\Phi$ and let  $\Phi^+$ be the set of positive    roots relative to $\Pi$. For each $\a\in\Phi^+$ let $\g_{\a}$ denote the corresponding root space of $\g$.   According to \cite[6.1]{j} we have  $\g=\mathfrak n^-+\mathfrak h+\mathfrak n^+$, where $$\mathfrak n^+=\sum_{\a\in\Phi^+}\g_{\a}, \quad \mathfrak n^-=\sum_{\a\in\Phi^+}\g_{-\a}.$$    Fix a  proper subset $I$ of $\Pi$ and put $\Phi_I=\mathbb ZI\cap\Phi$ and $ \Phi^+_I=\Phi_I\cap \Phi^+.$  Define   $\tilde\g_I=\mathfrak h+\sum_{\a\in\Phi_{I}}\g_{\a}$, as well as
$$\mathfrak u=\sum_{\a\in \Phi^+\setminus \Phi^+_I}\g_{\a},\quad \mathfrak u'=\sum_{\a\in \Phi^+\setminus \Phi^+_I}\g_{-\a}.$$ Then $\mathfrak p_I=\tilde\g_I+\mathfrak u$  is a parabolic subalgebras of $\g$ with Levi factor $\tilde\g_I$ \cite[10.6]{j}.
  Throughout the paper we assume  that $\chi (\mathfrak n^+)=0$. This is done without loss of generality due to \cite[Lemma 6.6]{j}.\par For any restricted Lie subalgebra $L$ of $\g$,  we denote by $u_{\chi}(L)$ the {\it $\chi$-reduced enveloping algebra} of $L$, where we continue to use $\chi$ for the restriction of $\chi$ to $L$ (\cite[5.3]{sf}).  If $\chi =0$,  then $u_{\chi}(L)$ is referred to as the {\it restricted enveloping algebra} of $L$, and denoted more simply  by $u(L)$. Let $L^{\chi}_I(\lambda)$ be a simple $u_{\chi}(\mathfrak p_I)$-module generated by a maximal vector $v_{\lambda}$ of weight $\lambda\in \mathfrak h^*$. Define the induced $u_{\chi}(\g)$-module $$Z^{\chi}_I(\lambda)=u_{\chi}(\g)\otimes _{u_{\chi}(\mathfrak p_I)}L^{\chi}_I(\lambda).$$ \par
     The main result of the present paper is Theorem 2.10, which  gives a sufficient condition for   $Z^{\chi}_I(\lambda)$ to be simple;  we show that $Z^{\chi}_I(\lambda)$ is simple  if $\lambda$ is not a zero of a certain polynomial $R^I_{\g}(\lambda)$. Then we give an application of this conclusion.

\section{Simplicity criterion}In this section, we keep the assumptions as in the introduction.  Let $$\{e_{\a}, h_{\beta}|\a\in\Phi, \beta\in\Pi\}$$ be a  Chevalley basis for $\g'=\text{Lie}(DG)$ such that $$[e_{\a}, e_{\beta}]=\pm (r+1)e_{\a+\beta},\quad \mathbin{\mathrm{if}}\quad\a,\ \beta, \ \a+\beta\in \Phi^+,$$ where $r$ is the greatest integer for which $\beta-r\a\in\Phi$ (see \cite[Theorem 25.2]{hu}). As pointed out in \cite[Section 3]{zh}, our assumption on $p$ ensures that $(r+1)\neq 0$.\par  For $\a\in \Phi^+$ put $f_{\a}=-e_{\a}$. Then we have $\g_{\a}=\mathbf Fe_{\a}$ and $\g_{-\a}=\mathbf Ff_{\a}$ for every $\a\in\Phi^+$. \par

 Remark: Let $\a,\beta\in\Phi^+$ such that $\a+\beta\in\Phi^+$ (resp. $\beta-\a\in\Phi^+$). Then we have $$[e_{\a}, e_{\beta}]=ce_{\a+\beta}, \quad [f_{\a}, f_{\beta}]=-cf_{\a+\beta} \ (\text{resp. $[e_{\a}, f_{\beta}]=cf_{\beta-\a}$})$$ for some $c\in\mathbf F\setminus 0$. For  brevity, we omit the scalar $c$. This does not affect any of the proofs in this section.\par
Recall from the introduction  the notation $\mathfrak p_I, \mathfrak u, \mathfrak u'$, $\tilde\g_I$, and $L^{\chi}_I(\lambda)$. Using (H3), we can show that $\mathfrak u$ is the nilradical of the parabolic subalgebra $\mathfrak p_I$.
By \cite[Corollary 1.3.8]{sf},  $L^{\chi}_I(\lambda)$ is  annihilated by $\mathfrak u$, and hence is a simple $u_{\chi}(\tilde\g_I)$-module.
 The simple $u_{\chi}(\mathfrak p_I)$-module $L^{\chi}_I(\lambda)$ is generated by any of  its maximal vectors, which may not  be unique under our assumption. \par    Recall from the introduction the induced $u_{\chi}(\g)$-module $Z_I^{\chi}(\lambda)$.  By the PBW theorem for the $\chi$-reduced enveloping algebra $u_{\chi}(\g)$ (\cite[Theorem 5.3.1]{sf}), we have $$Z_I^{\chi}(\lambda)\cong u_{\chi}(\mathfrak u')\otimes_{\mathbf F}L^{\chi}_{I}(\lambda)$$ as $u_{\chi}(\mathfrak u')$-modules. \par
    Let $\Phi^+\setminus \Phi^+_I=\{ \beta_1,\beta_2,\dots, \beta_k\}$, and let $v_1,\dots, v_n$ be a basis of $L^{\chi}_I(\lambda)$.  Then $Z^{\chi}_I(\lambda)$ has a basis $$f_{\beta_1}^{l_1}f_{\beta_2}^{l_2}\cdots f_{\beta_k}^{l_k}\otimes v_j,\quad 0\leq l_i\leq p-1,\ i=1,\dots,k,\ j=1,\dots, n.$$
 Let $\mathfrak S$ be a subset of $\Phi^+$. We say that $\mathfrak S$ is a {\it closed subset} if $\a+\beta\in \mathfrak S$ for any $\a, \beta\in \mathfrak S$ such that $\a+\beta\in \Phi^+$. Therefore, $\Phi^+\setminus \Phi^+_I$ is a closed subset of $\Phi^+$. We see that $\mathfrak S$ is a closed subset of $\Phi^+$ if and only if \ $\mathfrak s=:\sum_{\a\in\mathfrak S}\g_{\a}$ \ is a Lie subalgebra of $\g$;  it is clear that $\mathfrak s$ is restricted.\par Let $\mathfrak S$ be a closed  subset of $\Phi^+.$ \ Applying almost verbatim Humphreys's argument
 in the proof of \cite[Lemma 1.4]{hu1}, we get the following result.
 \begin{lemma} Let $(\a_1,\dots, \a_m)$ be any ordering of $\mathfrak S$. If $ht(\a_k)=h$, assume that all exponents $i_j$ in \ $e_{\a_1}^{i_1}\cdots e_{\a_m}^{i_m}\in u(\mathfrak s)$ \ for which $ht(\a_j)\geq h$ are equal to $p-1$. Then, if $e_{\a_k}$ is inserted anywhere into this expression, the result is $0$.
 \end{lemma}

\begin{corollary} Let $\Phi^+\setminus\Phi^+_I=\{\beta_1, \dots, \beta_k\}$. For $e_{\beta_1}^{p-1}\cdots e_{\beta_k}^{p-1}\in u(\mathfrak u)$, we have in $u_{\chi}(\g)$ that
  $$(1)\ [e_{\a}, e_{\beta_1}^{p-1}\cdots e_{\beta_k}^{p-1}]=0, \ (2)\ [f_{\a}, e_{\beta_1}^{p-1}\cdots e_{\beta_k}^{p-1}]=0$$ for every $\a\in \Phi^+_I$.
\end{corollary} \begin{proof} (1) is immediate from Lemma 2.1.\par (2) is given by \cite[Lemma 3.5]{zh}.
\end{proof}
In the following we assume $\beta_1, \dots, \beta_k$ in $\Phi^+\setminus \Phi_I^+$ is in the order of ascending heights. By \cite[Theorem 3.3.1]{sf}, $u_{\chi}(\mathfrak u')$ has a basis $$f_{\beta_1}^{l_1}\cdots f_{\beta_k}^{l_k},\quad 0\leq l_i\leq p-1.$$
\begin{lemma}  Let $1\leq s\leq k$, and let $f_{\beta_s}^{l_s}\cdots f_{\beta_k}^{l_k}\in u_{\chi}(\mathfrak u'),\ 0\leq l_s, \dots, l_k\leq p-1$. For any $j\geq s$, the product $f_{\beta_j}f_{\beta_s}^{l_s}\cdots f_{\beta_k}^{l_k}$ equals a linear combination of elements $f_{\beta_s}^{i_s}\cdots f_{\beta_k}^{i_k}$ such that $0\leq i_s, \dots, i_k\leq p-1$.
\end{lemma}
\begin{proof} For $\beta_i, \beta_j$ in $\Phi^+\setminus \Phi^+_I$,
if $\beta_i+\beta_j$ is also a root, then its height is greater than those of both $\beta_i$ and $\beta_j$. This implies that
$$\mathfrak u'_s=:\langle f_{\beta_s}, \dots, f_{\beta_k}\rangle$$ is a (restricted) Lie subalgebra of $\mathfrak u'$.  Then \cite[Theorem 3.3.1]{sf} says that $u_{\chi}(\mathfrak u'_s)$,  viewed as a subalgebra of $u_{\chi}(\mathfrak u')$,   has a basis $$f_{\beta_s}^{l_s}\cdots f_{\beta_k}^{l_k},\quad 0\leq l_s, \dots, l_k\leq p-1.$$ Then the lemma follows.
\end{proof}
\begin{lemma} For  $\a\in I$,   the element $[e_{\a}, f_{\beta_1}^{p-1}\cdots f_{\beta_k}^{p-1}]\in u_{\chi}(\g)$ equals a linear combination of basis vectors  $$f_{\beta_1}^{l_1}\cdots f_{\beta_k}^{l_k},\quad 0\leq l_i\leq p-1$$ of $u_{\chi}(\mathfrak u')$ such that $l_i<p-1$ for some $i=1,\dots, k$.\end{lemma}
\begin{proof}
If $[e_{\a}, f_{\beta_i}]\neq 0$ for some $i$, then since $\a$ is a simple root, we have $$[e_{\a}, f_{\beta_i}]=f_{\beta_i-\a}$$ with $\beta_i-\a=\beta_j$ for some $j=1,\dots, k$.  Clearly we have $\text{ht} (\beta_i-\a)<\text{ht}{\beta_i}$. For  $$[e_{\a}, f_{\beta_1}^{p-1}\cdots f_{\beta_k}^{p-1}]=\sum _{i,\ [e_{\a}, f_{\beta_i}]\neq 0} f_{\beta_1}^{p-1}\cdots f_{\beta_i-\a}^{p-1}\cdots (\sum_{s=0}^{p-2}f_{\beta_i}^sf_{\beta_i-\a}f_{\beta_i}^{p-2-s})\cdots f_{\beta_k}^{p-1},$$
each summand equals $$(p-1)f_{\beta_1}^{p-1}\cdots f_{\beta_i-\a}^{p}\cdots f_{\beta_i}^{p-2}\cdots f_{\beta_k}^{l_k}$$$$+\sum_{s=0}^{p-2}f_{\beta_1}^{p-1}\cdots f_{\beta_i-\a}^{p-1}[\cdots f_{\beta_i}^s, f_{\beta_i-\a}]f_{\beta_i}^{p-2-s}\cdots f_{\beta_k}^{l_k}.$$
Since $f_{\beta_i-\a}^p= \chi(f_{\beta_i-\a})^p$, the first term above is of the desired form.  For each term in the summation above, if there is $\beta_t$ between  $\beta_i-\a$ and $\beta_i$ (including $\beta_i$) such that $$f_{\beta}=:[f_{\beta_t}, f_{\beta_i-\a}]\neq 0,$$
so that $\text{ht}\beta >\text{ht}\beta_t$, then by Lemma 2.3 this term  equals a linear combination of elements $f_{\beta_1}^{l_1}\cdots f_{\beta_t}^{l_t}\cdots f_{\beta_k}^{l_k}$, $0\leq l_i\leq p-1$, with especially $l_t<p-1$. Thus, the lemma holds.

\end{proof}
\begin{lemma}For $\a\in I$,  $[f_{\a}, f_{\beta_1}^{p-1}\cdots f_{\beta_k}^{p-1}]\in u_{\chi}(\g)$ equals a linear combination of of basis vectors  $$f_{\beta_1}^{l_1}\cdots f_{\beta_k}^{l_k},\quad 0\leq l_i\leq p-1$$ of $u_{\chi}(\mathfrak u')$ such that $l_i<p-1$ for some $i=1,\dots, k$.\end{lemma}
\begin{proof}If $[f_{\a}, f_{\beta_i}]\neq 0$ for some $i$, so that $[f_{\a}, f_{\beta_i}]=f_{\beta_j}$ for some $j>i$ since $\text{ht}\beta_j>\text{ht}\beta_i$. Applying Lemma 2.3 we see that $$f_{\beta_1}^{p-1}\cdots [f_{\a}, f_{\beta_i}^{p-1}]\cdots f_{\beta_k}^{p-1}$$ is a linear combination of $f_{\beta_1}^{l_1}\cdots f_{\beta_i}^{l_i}\cdots f_{\beta_k}^{l_k}$ with $0\leq l_1,\dots, l_i,\dots, l_k\leq p-1$, and in particular, $l_i=p-2<p-1$. Then the lemma follows.
\end{proof}
By a similar proof as that for \cite[Lemma 3.6]{zh}, we obtain the following lemma.
\begin{lemma} There is a uniquely determined scalar $R^I_{\g}(\lambda)\in \mathbf F$ such that
$$e_{\beta_1}^{p-1}\cdots e_{\beta_k}^{p-1}f_{\beta_1}^{p-1}\cdots f_{\beta_k}^{p-1}\otimes v_{\l}=R^I_{\g}(\l)\otimes v_{\lambda}$$ in $Z^{\chi}_I(\lambda)$.
\end{lemma}

Let $\chi\in \g^*$ as given earlier.
 Then $\chi$ can be written as $\chi=\chi_s+\chi_n$, with $\chi_s(\mathfrak n^++\mathfrak n^-)=0$ and $\chi_n(\mathfrak h+\mathfrak n^+)=0$.
 For each simple $u_{\chi_s}(\mathfrak p_I)$-module $L^{\chi_s}_I(\lambda)$ ($\lambda\in \mathfrak h^*$), define the induced module $$Z^{\chi_s}_I(\lambda)=u_{\chi_s}(\g)\otimes_{u_{\chi_s}(\mathfrak p_I)}L^{\chi_s}_I(\lambda).$$ Let $v_{\l}\in L^{\chi_s}_I(\lambda)$ be a maximal vector of weight $\lambda$. In a similar way as above we define the scalar $R^I_{\g}(\lambda)_s$ by $$e_{\beta_1}^{p-1}\cdots e_{\beta_k}^{p-1}f_{\beta_1}^{p-1}\cdots f^{p-1}_{\beta_k}\otimes v_{\lambda}=R^I_{\g}(\lambda)_s\otimes v_{\lambda}.$$ We have
 $R^I_{\g}(\lambda)_s=R^I_{\g}(\lambda)$ for any $\lambda\in \mathfrak h^*$ by a similar proof as that for  \cite[Lemma 4.1]{zh}. Therefore, in calculating $R^I_{\g}(\lambda)$, we may assume $\chi=\chi_s$.
 By \cite[Lemma 4.1]{zh}, there is nonzero $c\in\mathbf F$ (independent of $\l$) such that $$R^I_{\g}(\lambda)=c\Pi^{k}_{i=1}[(\lambda+\rho)(h_{\beta_i})^{p-1}-1],\quad \l\in \mathfrak h^*.$$

We now put the basis elements  of $\mathfrak n^+$ in the order $$e_{\a_1}, \dots, e_{\a_s}, e_{\beta_1}, \dots, e_{\beta_k},$$ for which $\{\a_1, \dots, \a_s\}=\Phi^+_I.$\par Let $U(\g)$ (resp. $U(\mathfrak h)$; $U(\mathfrak n^+)$) be the universal enveloping algebra of $\mathfrak g$ ($\mathfrak h$; $\mathfrak n^+$). Then $U(\g)$ is naturally a $T$-module with the adjoint action. Let us denote the $T$-weight of a weight vector $u\in U(\g)$ by $\text{wt} (u)$.\par
 Assume $l=(l_1,\dots, l_k)\in \mathbb N^k$ with $l_i\leq p-1$.  Then we have in $U(\g)$ $$x=e_{\beta_1}^{p-1}\cdots e_{\beta_k}^{p-1}f_{\beta_1}^{l_1}\cdots f_{\beta_k}^{l_k}=\sum _i u_i^-u^0_iu^+_i,\quad u^-_i\in U(\mathfrak n^-),\ u^0_i\in U(\mathfrak h),\ u^+_i\in U(\mathfrak n^+).$$
Under the $T$-action, the weight of $x$ is $$\text{wt}(x)=\text{wt}(u_i^+) +\text{wt}(u^-_i)=\sum_{i,\ l_i<p-1}(p-1-l_i)\beta_i.$$
\begin{lemma}Keep the assumptions above. If one of the exponents $l_1,\dots, l_k$ is strictly less than $p-1$, then for each $i$ we have $$u_i^+=e_{\a_1}^{j_1}\cdots e_{\a_s}^{j_s}e_{\beta_1}^{t_1}\cdots e_{\beta_k}^{t_k},\quad  j_1, \dots, j_s,\ t_1,\dots, t_k\in \mathbb N$$ such that \ $t_q>0$ for some $q=1,\dots, k$.
\end{lemma}
\begin{proof} Suppose otherwise that $t_q=0$ for all $q$. Then we have  $$\text{wt}(u^+_i)=\sum_{i=1}^s j_i\a_i\in \mathbb N I.$$ Note that   $$\text{wt}(u^+_i)=-\text{wt}(u^-_i)+\sum_{i,\ l_i<p-1}(p-1-l_i)\beta_i.$$
Expressed as linear combinations of simple roots $\Pi$, $-\text{wt}(u^-_i)$ has all coefficients nonnegative,
but $$\sum_{i,\ l_i<p-1}(p-1-l_i)\beta_i$$ has at least a positive coefficient for some simple root in $\Pi\setminus I$,  a contradiction.

\end{proof}
\begin{lemma}Keep the assumptions above.  If one of the exponents $l_1,\dots, l_k$ is strictly less than $p-1$,  then we have
in $Z^{\chi}_I(\l)$ that $$e_{\beta_1}^{p-1}\dots e_{\beta_k}^{p-1} f_{\beta_1}^{l_1}\cdots f_{\beta_k}^{l_k}\otimes v=0\quad \mathbin{\mathrm{for \ any}}\ v\in L^{\chi}_I(\l).$$
\end{lemma}
\begin{proof} By our assumptions, we have in $U(\g)$ that $$e_{\beta_1}^{p-1}\dots e_{\beta_k}^{p-1} f_{\beta_1}^{l_1}\cdots f_{\beta_k}^{l_k}=\sum  u^-_i u^0_i u^+_i,\ u^-_i\in U(\mathfrak n^-),\ u^0_i\in U(\mathfrak h),\ u^+_i\in U(\mathfrak n^+),$$ where by Lemma 2.7 each $u^+_i$ is of the form $$e_{\a_1}^{j_1}\cdots e_{\a_s}^{j_s}e_{\beta_1}^{t_1}\cdots e_{\beta_k}^{t_k}$$ with $t_q>0$ for some $q$. Then we have in $u_{\chi}(\g)$ that $$e_{\beta_1}^{p-1}\dots e_{\beta_k}^{p-1} f_{\beta_1}^{l_1}\cdots f_{\beta_k}^{l_k}=\sum  \bar u^-_i \bar u^0_i \bar u^+_i, \quad \bar u^-_i\in u_{\chi}(\mathfrak n^-),\ \bar u^0_i\in u_{\chi}(\mathfrak h),\ u^+_i\in u(\mathfrak n^+).$$
For each $u_i^+$ of the  above form,  we have $\bar u^+_i=0$ if one of the exponents $$j_1, \dots, j_s, t_1,\dots, t_k$$ is greater than or equal to $p$. If all these exponents are less than $p$, then $\bar u^+_i$ is a basis vector of $u(\mathfrak n^+)$ of the same form as $u^+_i$, and hence $\bar u^+_i\otimes v=0$ since one of the $t_q$ is positive.  Therefore we have   $\bar u^+_i\otimes v=0$ for all $i$, so the lemma follows.
\end{proof}

\begin{lemma} Keep the assumptions above. Let $v\in L^{\chi}_I(\l)$. For any $\a\in I$, we have in $Z^{\chi}_I(\l)$ that
$$(1) \quad e_{\a}e_{\beta_1}^{p-1}\cdots e_{\beta_k}^{p-1}f_{\beta_1}^{p-1}\cdots f_{\beta_k}^{p-1}\otimes v=e_{\beta_1}^{p-1}\cdots e_{\beta_k}^{p-1}f_{\beta_1}^{p-1}\cdots f_{\beta_k}^{p-1}\otimes e_{\a}v,$$
$$(2) \quad f_{\a}e_{\beta_1}^{p-1}\cdots e_{\beta_k}^{p-1}f_{\beta_1}^{p-1}\cdots f_{\beta_k}^{p-1}\otimes v=e_{\beta_1}^{p-1}\cdots e_{\beta_k}^{p-1}f_{\beta_1}^{p-1}\cdots f_{\beta_k}^{p-1}\otimes f_{\a}v.$$
\end{lemma}
\begin{proof} (1) By Corollary 2.2, the left side equals $$e_{\beta_1}^{p-1}\cdots e_{\beta_k}^{p-1}e_{\a}f_{\beta_1}^{p-1}\cdots f_{\beta_k}^{p-1}\otimes v=e_{\beta_1}^{p-1}\cdots e_{\beta_k}^{p-1}f_{\beta_1}^{p-1}\cdots f_{\beta_k}^{p-1}\otimes e_{\a}v$$$$+e_{\beta_1}^{p-1}\cdots e_{\beta_k}^{p-1}[e_{\a}, f_{\beta_1}^{p-1}\cdots f_{\beta_k}^{p-1}]\otimes v$$ $$=e_{\beta_1}^{p-1}\cdots e_{\beta_k}^{p-1}f_{\beta_1}^{p-1}\cdots f_{\beta_k}^{p-1}\otimes e_{\a}v,$$ where the last equality follows from Lemma 2.4 and 2.8.\par
(2) follows from Corollary 2.2, Lemma 2.5 and 2.8.
\end{proof}
\begin{theorem} The $u_{\chi}(\g)$-module $Z^{\chi}_I(\lambda)$ is simple if $R^I_{\g}(\lambda)\neq 0$.
\end{theorem}
\begin{proof} Suppose $R^I_{\g}(\lambda)\neq 0$. Let $N$ be a nonzero submodule of $Z_I^{\chi}(\lambda)$, and let $x\in N$ be a nonzero element,  which we can write $$x=\sum_l c_l f_{\beta_1}^{l_1}\cdots f_{\beta_k}^{l_k}\otimes v_l, $$ where the sum is over all tuples $l=(l_1,\dots,l_k)$ with $
0\leq l_i\leq p-1$ and where $c_l\in\mathbf F$ and $0\neq v_l\in L^{\chi}_I(\l)$.\par First, applying  $f_{\beta_i}$'s with $\chi(f_{\beta_i})=0$ in the order of descending heights, then applying  $f_{\beta_i}$'s with $\chi(f_{\beta_i})\neq 0$ in the order of ascending heights,  we get $$x'=\sum c_lf_{\beta_1}^{l_1}\cdots f_{\beta_k}^{l_k}\otimes v_l\in N,\quad 0\leq l_i\leq p-1, $$ where for each $l=(l_1, \dots, l_k)$ with $c_l\neq 0$, we have $l_j=p-1$ if $\chi(f_{\beta_j})=0$,
but we may have $l_i<p-1$ if $\chi(f_{\beta_i})\neq 0$.
 In particular, we have $c_{\hat l}\neq 0$ for $$\hat l=(p-1,\dots, p-1).$$ \par It follows that $$e_{\beta_1}^{p-1}\cdots e_{\beta_k}^{p-1}x'=\sum_l c_l e_{\beta_1}^{p-1}\dots e_{\beta_k}^{p-1} f_{\beta_1}^{l_1}\cdots f_{\beta_k}^{l_k}\otimes v_l\in N.$$

If $l\neq \hat l$, that is, $l_i<p-1$ for some $\beta_i$, then we have by Lemma 2.8 that $$c_l e_{\beta_1}^{p-1}\dots e_{\beta_k}^{p-1} f_{\beta_1}^{l_1}\cdots f_{\beta_k}^{l_k}\otimes v_l=0,$$ and hence
$$e_{\beta_1}^{p-1}\cdots e_{\beta_k}^{p-1}x'=c_{\hat l}e_{\beta_1}^{p-1}\cdots e_{\beta_k}^{p-1}f_{\beta_1}^{p-1}\cdots f_{\beta_k}^{p-1}\otimes v_{\hat l}\in N.$$
Let $v_{\l}\in L^{\chi}_I(\l)$ be a maximal vector  of weight $\l$. Since $L^{\chi}_I(\l)$ is simple as a $u_{\chi}(\tilde\g_I)$-module, there is $$u=\sum u_i^-u^0_iu^+_i\in u_{\chi}(\tilde \g_I)$$ such that $uv_{\hat l}=v_{\l}$. It is clear that $$[h, e_{\beta_1}^{p-1}\cdots e_{\beta_k}^{p-1}f_{\beta_1}^{p-1}\cdots f_{\beta_k}^{p-1}]=0$$ for all $h\in\mathfrak h$. Then by Lemma 2.9 and 2.6 we have $$\begin{aligned} ue_{\beta_1}^{p-1}\cdots e_{\beta_k}^{p-1}f_{\beta_1}^{p-1}\cdots f_{\beta_k}^{p-1}\otimes v_{\hat l}&=e_{\beta_1}^{p-1}\cdots e_{\beta_k}^{p-1}f_{\beta_1}^{p-1}\cdots f_{\beta_k}^{p-1}\otimes uv_{\hat l}\\&=e_{\beta_1}^{p-1}\cdots e_{\beta_k}^{p-1}f_{\beta_1}^{p-1}\cdots f_{\beta_k}^{p-1}\otimes v_{\l}\\ &= R^I_{\g}(\l)\otimes v_{\l}\in N,\end{aligned}$$ implying that $1\otimes v_{\l}\in N$, and hence $N=Z^{\chi}_I(\l)$.  Thus, $Z_I^{\chi}(\lambda)$ is simple.\par

\end{proof}
 In \cite[5.1]{fp2}, Friedlander and Parshall also asked, for $\chi\in\g^*$,  which $u_{\chi}(\g)$-modules are induced from a restricted module for some proper restricted subalgebra of $\g$. \cite[Theorem 1.2]{lsy2} partially answers this question. We now study this question as an application of the
 above theorem.\par
   Keep the notation  from the introduction. We assume that $\chi (\mathfrak h+\mathfrak n^+)=0$. Assume that there is a subset $I$ of $\Pi$ such that $\chi (\mathfrak p_I)=0$. Then $Z^{\chi}_I(\lambda)$ is induced from a restricted simple $\mathfrak p_I$-module.\par
    For example, if $\chi$ is in the standard Levi form, i.e.,
   there is a subset $J\subseteq \Pi$ such that $\chi (f_{\a})\neq 0$ for all $\a\in J$, and $\chi (f_{\a})=0$ for all $\a\in \Phi^+\setminus J$, we take $I=\Pi\setminus J$.\par
   Recall the maximal torus $T$ of $G$ from the introduction. Let $X(T)$ be  the character group of $T$. A weight
   $\tilde\mu\in X(T)$ is $p$-regular if its stabilizer in the affine Weyl group of $\g$ is trivial. By \cite[p.225]{j1}, the differential of $\tilde\mu$ is an element in $\mu\in\mathfrak h^*$ satisfying $\mu (h^{[p]})=\mu (h)^p$ for all \ $h\in\mathfrak h$. \par
  Now let $\lambda\in \mathfrak h^*$ be the differential of a weight $\tilde \lambda\in X(T)$ such that $R^I_{\g}(\lambda)\neq 0$. Using the notation from Section 2, we have $$\Phi^+ \setminus \Phi^+_I=\{\beta_1,\dots, \beta_k\}.$$ Then $R^I_{\g}(\lambda)\neq 0$ implies that $(\lambda+\rho)(h_{\beta_i})=0$ for all $i$, so that the weight $\tilde\lambda$ is not $p$-regular.
   But Theorem 2.10 says that $Z^{\chi}_I(\lambda)$ is simple.   Therefore, $Z^{\chi}_I(\lambda)$ is not included in those constructed in \cite[Theorem 1.2]{lsy2} in the case that $\chi$ is in the standard Levi form.

\def\refname{
\end{document}